\documentclass[preprint,12pt]{elsarticle}
\usepackage{amssymb}
\usepackage{amsfonts,latexsym}
\usepackage{epsfig}
\usepackage{amsmath}
\usepackage{amsthm}
\usepackage{epstopdf}
\usepackage{graphicx}
\usepackage{amsmath}
\usepackage{color}
\usepackage{ulem}
\usepackage{tikz}
\usepackage{algorithm}
\usepackage{algpseudocode}
\usepackage{tabularx}
\usepackage{paralist}
\usepackage{algorithmicx}
\usepackage{booktabs}
\usepackage{mathtools}
\usepackage[scr=boondoxupr]{mathalpha}

\usetikzlibrary{decorations.markings}
\usepackage[top=3cm,bottom=4cm,right=3cm,left=3cm]{geometry}
\newtheorem{idea}{Idea}
\journal{Applied Numerical Mathematics}
\begin{document}

\begin{frontmatter}

\title{New preconditioner strategy for solving  block four-by-four linear systems: An application to the saddle-point problem from 3D Stokes equation}

\author[1]{A. Badahmane}
\author[1]{A. Ratnani}
\author[2]{H. Sadok}
\address[1]{The UM6P Vanguard Center, Mohammed VI Polytechnic University, Benguerir 43150, Lot 660, Hay Moulay Rachid, Morocco.}
\address[2]{LMPA, Universit\'e du Littoral C\^ote d'Opale, 50 Rue F. Buisson, BP 699 - 62228 Calais cedex, France.}

\newcommand{\Leg}{{\mathcal L(E,G)}}
\newcommand{\Lef}{{\mathcal L(E,F)}}
\newcommand{\Lfg}{{\mathcal L(F,G)}}
\newcommand{\Le}{{\mathcal L(E)}}
\newcommand{\I}{{\mathcal{I}}}
\newcommand{\ia}{{\mathfrak I}}
\newcommand{\vi}{\emptyset}
\newcommand{\di}{\displaystyle}
\newcommand{\Om}{\Omega}
\newcommand{\na}{\nabla}
\newcommand{\wi}{\widetilde}
\newcommand{\al}{\alpha}
\newcommand{\be}{\beta}
\newcommand{\ga}{\gamma}
\newcommand{\Ga}{\Gamma}
\newcommand{\e}{\epsilon}
\newcommand{\la}{\lambda}
\newcommand{\De}{\Delta}
\newcommand{\de}{\delta}
\newcommand{\entraine}{\Longrightarrow}
\newcommand{\inj}{\hookrightarrow}
\newcommand{\recip}{\Longleftarrow}
\newcommand{\ssi}{\Longleftrightarrow}
\newcommand{\K}{\mathbbm{K}}
\newcommand{\A}{\mathcal{A}}
\newcommand{\R}{\mathbb{R}}
\newcommand{\C}{\mathbb{C}}
\newcommand{\N}{\mathbb{N}}
\newcommand{\Q}{\mathbb{Q}}
\newcommand{\1}{\mathbb{1}}
\newcommand{\0}{\mathbb{0}}
\newcommand{\Z}{\mathbbm{Z}}
\newcommand{\E}{\mathbbm{E}}
\newcommand{\F}{\mathbbm{F}}
\newcommand{\B}{\mathbbm{B}}
\newcommand{\M}{\mathcal{M}_{n}(\K)}
\newcommand{\tend}[2]{\displaystyle\mathop{\longrightarrow}_{#1\rightarrow#2}}

\font\bb=msbm10

\def\ent{{{\rm Z}\mkern-5.5mu{\rm Z}}}
\newtheorem{exo}{Exercice}

\newtheorem{pre}{Preuve}
\newtheorem{pro}{Propriété}
\newtheorem{exe}{Example}
\newtheorem{theorem}{Theorem}[section]
\newtheorem{proposition}{Proposition}
\newtheorem{definition}{Definition}[section]
\newtheorem{remark}{Remark}[section]
\newtheorem{lem}{Lemma}[section]
\begin{abstract}
We have presented a fast method for solving a specific type of block four-by-four saddle-point problem arising from the finite element discretization of the generalized 3D Stokes problem. We analyze the eigenvalue distribution and the eigenvectors of the preconditioned matrix.
Furthermore, we suggested utilizing  the preconditioned global conjugate gradient method  ($\mathcal{P}$GCG) as a block iterative solver for handling multiple right-hand sides within the sub-system and  give some new convergence results. Numerical experiments have shown that our preconditioned iterative approach is very efficient  for solving the  3D Stokes problem.
\end{abstract}
\end{frontmatter}
\textbf{Notations.}
 For a real square  matrix $A$, the set of all eigenvalues (spectrum) of $A$ is denoted by $\sigma(A)$. When the spectrum of $A$ is real, we use $\lambda_{\text{min}}(A)$ and $\lambda_{\text{max}}(A)$ to respectively denote its minimum and maximum eigenvalues. When $A$ is symmetric positive (semi)definite, we write $A \succ 0$ ($A \succeq 0$). In addition, for two given matrices $A$ and $C$, the relation $A \succ C$ ($A \succeq C$) means $A - C \succ 0$ ($A - C \succeq 0$). Lastly, for vectors $\mathrm{u}_1$, $\mathrm{u}_2$, $\mathrm{u}_3$ and  $\mathrm{p}$ of dimensions $n_u$ and $n_p$, $\left(\mathrm{u}_1;\mathrm{u}_2;\mathrm{u}_3;\mathrm{p}\right)$ will denote a column vector of dimension $N=3n_u +n_p$. Throughout the paper, $I$ will denote the identity matrix (the size of which will be clear from the context).
\section{Introduction}\label{sec1}
\subsection{Problem setting and variational formulations}~\label{problem}
Let \(\Omega\) represent a bounded domain in \(\mathbb{R}^3\), with a boundary \(\Gamma\) that is Lipschitz-continuous. Within \(\Omega\), consider the generalized Stokes problem:
\begin{eqnarray}
\label{Stokes}
\left\{
\begin{array}{r c l}
 \alpha \vec{u}-\nu\Delta \vec{u}+\vec{\nabla} p &= \vec{f}\hspace*{0.3cm} \text{in}  \hspace*{0.1cm}
\Omega, \label{Stokes1} \\
\vec{\nabla}\cdot \vec{u}&=0 \hspace*{0.4cm} \text{in}  \hspace*{0.2cm} \Omega, \label{Stokes2}\\
\vec{u}&=\vec{u}_{D} \hspace*{0.2cm} \text{on}  \hspace*{0.1cm}\Gamma.
\end{array}
\right.
\end{eqnarray}
In our research paper, we consider a system of equations involving the velocity vector $\Vec{u}=\left(u_{1};u_{2};u_{3}\right)\in\mathbb{R}^{3}$, $\nu>0$ represents the kinematic viscosity, pressure $p$, external force field $\Vec{f}=(f_{1};f_{2};f_{3})\in\mathbb{R}^{3}$ and  $\alpha$ is a non-negative constant. When $\alpha=0$, the equations reduce to the following  classical Stokes problem :
\begin{eqnarray}
\label{StokesC}
\left\{
\begin{array}{r c l}
-\nu\Delta \vec{u}+\vec{\nabla} p &= \vec{f}\hspace*{0.3cm} \text{in}  \hspace*{0.1cm}
\Omega,  \\
\vec{\nabla}\cdot \vec{u}&=0 \hspace*{0.4cm} \text{in}  \hspace*{0.2cm} \Omega,\\
\vec{u}&=\vec{u}_{D} \hspace*{0.2cm} \text{on}  \hspace*{0.1cm}\Gamma.
\end{array}
\right.
\end{eqnarray}
However for $\alpha>0$,~(\ref{Stokes}) represent the generalized Stokes problem. Before starting the weak formulation of  Eq.~(\ref{Stokes}), we provide some definition and reminders:
the space of functions that are
square-integrable according to  Lebesgue definition is a set of functions  where the integral of the square of the function over a given interval is finite, and also can be expressed as follows :
$$L_{2}(\Omega):=\left\{e:\Omega \rightarrow\mathbb{R} \hspace{0.1cm}
 ;\int_{\Omega}e^{2}<\infty\right\},$$
if we have a subset $\Omega$ of the three-dimensional Euclidean space $\mathbb{R}^3$, then the Sobolev space $\mathcal{H}^{1}(\Omega)$ can be defined as follows:
$$\mathcal{H}^{1}(\Omega)=\left\{e:\Omega \rightarrow  \mathbb{R}\hspace{0.1cm};e, \frac{\partial e}{\partial x},\frac{\partial e}{\partial y},\frac{\partial e}{\partial z} \in L_{2}(\Omega)\right\}.$$ 
We define the velocity solution and test spaces:
$$\mathbb{V}^{D}=\left\{\vec{u}\in \mathcal{H}^{1}(\Omega)^{3}\hspace{0.1cm};  \vec{u}=\vec{u}_{D} \hspace{0.1cm} \text{on} \hspace{0.1cm}\partial\Omega\right\},$$
$$H^{1}_{E_0}:=\left\{\vec{v}\in \mathcal{H}^{1}(\Omega)^{3}\hspace{0.1cm};  \vec{v}=\vec{0} \hspace{0.1cm} \text{on} \hspace{0.1cm}\partial\Omega\right\},$$
$$\mathbb{P}=\left\{p\in L_{2}(\Omega)\hspace{0.1cm}; \int_{\Omega}pdx=0\right\},$$ the variational formulation of  $(\ref{Stokes})$,
find  $\vec{u}\in \mathbb{V}^{D}$ and $p\in \mathbb{P}$ such that :
\begin{eqnarray}
\label{var}
\left\{
\begin{array}{r c l}
\displaystyle{\alpha \int_{\Omega }^{}\vec{u}\cdot\vec{v}+\nu\int_{\Omega }^{}\vec{\nabla} \vec{u}:\vec{\nabla} \vec{v}-\int_{\Omega }^{}p\vec{\nabla}\cdot \vec{v} }&=&\displaystyle{ \int_{\Omega }^{}\vec{f}\cdot  \vec{v}}\hspace{0.2cm}\  \text{for all}\hspace{0.1cm} \vec{v}\in H^{1}_{E_0}, \\
 \displaystyle{\int_{\Omega }^{}q\vec{\nabla}\cdot \vec{u}}&=& 0 \hspace{1.6cm} \text{for all}\hspace{0.1cm} q\in \mathbb{P}.
\end{array}
\right.
\end{eqnarray}
Here  $\cdot$ is the scalar product and  $\vec{\nabla} \vec{u} : \vec{\nabla} \vec{v}$ represents the component-wise scalar product. For instance, in three dimensions, it can be represented as $\vec{\nabla} u_{x}\cdot\vec{\nabla} v_{x}+\vec{\nabla} u_{y}\cdot\vec{\nabla} v_{y}+\vec{\nabla} u_{z}\cdot\vec{\nabla} v_{z}$. 
\subsection{Finite element discretization with $P_{1}$-Bubble/$P_{1}$} \cite{koko}
In this paper,  we used $P_{1}$-bubble/$P_{1}$ to discretize  (\ref{Stokes}). Let $\mathcal{T}_{\mathrm{h}}$ be a triangulation of $\Omega$, $T$ a triangle of $\mathcal{T}_{\mathrm{h}}$ and $\mathcal{C}^0$ represents the space of continuous functions on a set $\Omega$. We define the space associated with the bubble by
\[  \mathbb{B}_{\mathrm{h}} = \{ v_{\mathrm{h}} \in \mathcal{C}^0(\overline{\Omega}) \,|\, \forall T \in \mathcal{T}_{\mathrm{h}}, \, v_{\mathrm{h}}|_T = xb^{(T)} \}. \]

We also define the discrete function spaces:
\[ \mathbb{V}_{i,\mathrm{h}} = \{ v_{i} \in \mathcal{C}^0(\overline{\Omega}) \,;\, v_{i}|_T \in P_1, \, \forall T \in \mathcal{T}_{\mathrm{h}}, \, v_{i}|_\Gamma = 0, \, i = 1, 2, 3 \}, \]
\[ \mathbb{P}_{\mathrm{h}} = \{ q_{\mathrm{h}} \in \mathcal{C}^0(\overline{\Omega}) \,;\, q_{\mathrm{h}}|_T \in P_1, \, \forall T \in \mathcal{T}_{\mathrm{h}}, \, \int_{\Omega} q_{\mathrm{h}} \,dx = 0 \}. \]
We set $\mathbb{X}_{i,{\mathrm{h}}} = \mathbb{V}_{i,{\mathrm{h}}} \oplus \mathbb{B}_{\mathrm{h}}$ and $\mathbb{X}_{\mathrm{h}} = X_{1,{\mathrm{h}}} \times X_{2,{\mathrm{h}}}\times X_{3,{\mathrm{h}}} $. With the above preparations, the discrete variational problem reads as follows : find $(\vec{u}_{\mathrm{h}}, p_{\mathrm{h}}) \in \mathbb{X}_{\mathrm{h}} \times \mathbb{P}_{\mathrm{h}}$ such that
\begin{eqnarray}
 \label{disc_var}
 \left\{
 \begin{array}{r c l}
 \displaystyle{\alpha \int_{\Omega }^{}\vec{u}_{{\mathrm{h}}}\cdot\vec{v}_{{\mathrm{h}}}+\nu\int_{\Omega }^{}\vec{\nabla}\vec{u}_{\mathrm{h}}:\vec{\nabla}\vec{v}_{\mathrm{h}}-\int_{\Omega }^{}p_{\mathrm{h}}\vec{\nabla}\cdot \vec{v}_{\mathrm{h}}} &=& \displaystyle{\int_{\Omega }^{}\vec{f}\cdot  \vec{v}_\mathrm{h}}\hspace{0.2cm}\  \text{for all}\hspace{0.1cm} \vec{v}_{\mathrm{h}}\in \mathbb{X}_{\mathrm{h}}, \\
 \displaystyle{\int_{\Omega }^{}q_{{\mathrm{h}}}\vec{\nabla} \cdot\vec{u}_{{\mathrm{h}}}}&=&0 \hspace{1.6cm}\text{for all}\hspace{0.1cm} {q}_{\mathrm{h}}\in \mathbb{P}_{{\mathrm{h}}}. \\ 
 \end{array}
 \right.
 \end{eqnarray}
For a given triangle $T$, the velocity field $\vec{u}_{\mathrm{h}}$ and the pressure $p_{\mathrm{h}}$ are approximated by linear combinations of the basis functions in the form
\[ \vec{u}_{{\mathrm{h}}}(x) = \sum_{i=1}^{4}\textbf{u}_i \phi_i(x)  + \textbf{u}_b \phi_b(x), \]
\[ p_{\mathrm{h}}(x) = \sum_{i=1}^{4} \textbf{p}_i\phi_i(x) ,\]
where $\textbf{u}_i$ and $\textbf{p}_i$ are nodal values of $\vec{u}_{\mathrm{h}}$ and $p_{h}$, while $\textbf{u}_b$ is the bubble value.  A useful property for the basis functions is
\[\phi_1(x) = 1-x-y,\hspace{0.2cm} \phi_2(x) = x, \hspace{0.2cm}\phi_3(x) =y,\hspace{0.2cm} \phi_4(x) = z, \hspace{0.2cm}\phi_b(x) =27\prod_{i=1}^{3} \phi_i(x)\]
. These formulas will be used to construct the coefficient matrices in the context of 3D finite element discretization.
To identify the corresponding linear algebra problem  Eq.~(\ref{saddle}), we introduce the  basis functions in 3D $\{\phi_{j}\}_{j=1,...,n_{u}}$, for more details we refer the reader to see \cite{Elman}, then $\vec{u}_{\mathrm{h}}$ and $p_{\mathrm{h}}$
can be expressed as follows:
\begin{eqnarray}
\label{vilos}
\vec{u}_{\mathrm{h}}=\sum_{j=1}^{4} \textbf{u}_{j} \phi_{j} +\textbf{u}_{b} \phi_{b},  \hspace {0.3cm}p_{\mathrm{h}}=\sum_{k=1}^{4} \textbf{p}_{k}\phi_{k},
\end{eqnarray}
and use them to formulate the problem in terms of linear algebra.
The algebraic formulation of Eq.~(\ref{disc_var}), can be expressed as a system of linear equations, which the matrix of the system is a saddle-point matrix defined as follows:
 \begin{eqnarray}
  \label{saddle}
\mathcal{A}\underbrace{\left( \begin{array}{c}
\mathrm{u}\\\mathrm{p}
\end{array}\right)}_{{x}}
= \left( \begin{array}{cccc}
 A & 0 &  0& B_{1}^{T} \\ 0 & A  & 0 &B_{2}^{T}  \\ 0 & 0 & A & B_{3}^{T}\\  B_{1} &  B_{2} & B_{3} & C
 \end{array} \right)\left( \begin{array}{c}
\mathrm{u}_{1} \\ \mathrm{u}_{2} \\ \mathrm{u}_{3}  \\  \mathrm{p} 
 \end{array}\right)=\underbrace{\left( \begin{array}{c}
\mathrm{f}_{1}\\ \mathrm{f}_{2} \\\mathrm{f}_{3} \\  g
 \end{array}\right)}_{d},
 \end{eqnarray}
where   $A$, $B$ and $C$ are given by:
 \begin{eqnarray}
 A&=&[a_{i,j}],\hspace{0.2cm} a_{i,j}=\alpha \int_{\Omega}\phi_{i}\phi_{j} +\int_{\Omega} \vec{\nabla}\phi_{i}:\vec{\nabla}\phi_{j}, \hspace{0.2cm} i,j=1,2,3, \nonumber\\
B_i&=&[b_{k,j}],\hspace{0.2cm} b_{k,j}=-\int_{\Omega} \psi_{k} \vec{\nabla}\cdot\phi_{j}, \hspace{0.2cm} j,k=1,2,3,\nonumber\\
C&=&\sum_{i=1}^{3} B_{ib}B_{ib}^{T}, \hspace{0.1cm}where \hspace{0.2cm}
B_{ib}=\frac{32}{105}|T|\begin{pmatrix}
\partial_i \phi_1 \\
\partial_i \phi_2 \\
\partial_i \phi_3 \\
\partial_i \phi_4 \\
\end{pmatrix},\nonumber
 \end{eqnarray}
and $|T|$ stands for the tetrahedron volume.
The paper is structured as follows. In section $2$  we  introduce  the global regularized  $\mathcal{P}_{Gr}$  preconditioner, used   for solving the  saddle-point problem $(\ref{saddle})$ arising within the weak formulation of generalized Stokes problem. The new convergence results of the preconditioned global conjugate gradient method will be analysed in section $3$. We present simulation results for 3D generalized Stokes problems obtained using the various preconditioners in Sec.~\ref{Numerical}. In conclusion, we provide  some recommendations 
on chosen  the appropriate preconditioner based on the characteristics of the  problem.
 \section{ Preconditioned Krylov subspace methods}\label{Pia}
Iterative methods are  suitable for solving   high-dimensional sparse linear systems, such as~(\ref{saddle}). One significant advantage of iterative methods  over direct methods is that they don’t  necessitate $\mathcal{A}$ to be factored. This hold true for Krylov subspace methods, including the $\mathcal{R}$GMRES and $\mathcal{F}$GMRES  methods implemented below. Nevertheless, drawback of iterative solvers are inexact, which necessitates a considerable  number of iterations to achieve the desired tolerance. To overcome this challenge,  preconditioning techniques are employed as a solution. In this context, the preconditioner $\mathcal{P}$ applies a linear transformation to system $(\ref{saddle})$ for reducing the condition number of the transformed matrix $\mathcal{P}^{-1}\mathcal{A}$. The resulting  preconditioned linear system can be represented as follows:
\begin{eqnarray}
\label{Psaddle}
\mathcal{P}^{-1}\mathcal{A}{x}=
\mathcal{P}^{-1}d.
\end{eqnarray}
Various implementations of  the preconditioned iterative approach can be devised based on the selection of the preconditioning linear transformation $\mathcal{P}$ and saddle-point matrix $\mathcal{A}$.  
When solving equation~\eqref{Psaddle} using any of the preconditioners  described in the following subsection, the iterations are stopped as soon as the Euclidean norm of the current residue is lower than a tolerance threshold  :
\begin{eqnarray}
\label{tau}
 \frac{\|\mathcal{P}^{-1}d-\mathcal{P}^{-1}\mathcal{A}x^{(k)}\|_{2}}{\|\mathcal{P}^{-1}d\|_{2}}<	\tau,
\end{eqnarray}
where $x^{(k)}$ denotes the current iterate and $\tau$ is the threshold value. Here $\|\cdot\|_{2}$ stands for the euclidean 
norm. The maximum number of iterations is $200$. We  allow  at most $100$ restarts. 
\subsection{Global regularized preconditioner}
Badahmane \cite{badahmane1}  introduced and examined a regularized preconditioner designed  for solving the nonsingular saddle-point problems of the form (\ref{saddle}),
where the coefficient matrix of the regularized preconditioner is given as follows:
 \begin{eqnarray}
  \label{Pr}
\mathcal{P}_{r}
=\left( \begin{array}{ccc|c}
 A & 0 &  0& B_{1}^{T} \\ 0 & A  & 0 &B_{2}^{T}  \\ 0 & 0 & A & B_{3}^{T}\\ 
 \hline 
 B_{1} & B_{2}  & B_{3}  & \beta Q
 \end{array} \right)=\left( \begin{array}{cc}
	A\otimes I_3 & B^{T} \\ B & \beta Q	\end{array} \right),
 \end{eqnarray}
 with $\beta>0$ and $Q\succ 0$. For example, for the discrete
Stokes system, $Q$ may be an approximation of the pressure Schur complement; see \cite{Elman,koko}. Each preconditioned  regularized $\mathcal{R}$GMRES iteration  requires the solution of a linear system of the following form :
\begin{eqnarray}
    \label{iteration}
\mathcal{P}_{r}\mathrm{z}=\mathcal{A}\mathrm{v}_{i},\hspace{0.1cm} i=1,...,k.
    \end{eqnarray}
Where $k$ denotes dimension of Krylov subspace, $\mathrm{v}_{i}\in\mathbb{R}^{N}$ are the basis vectors of Krylov subspace $\mathcal{K}_{k}\left(\mathcal{P}^{-1}_{r}\mathcal{A},\mathcal{P}^{-1}_{r}R_{0}\right)$, where $R_0=d-\mathcal{A}x_0$,\\ 
$\mathcal{K}_{k}\left(\mathcal{P}^{-1}_{r}\mathcal{A},   \mathcal{P}^{-1}_{r}R_{0}\right)=\text{span}\{\mathcal{P}^{-1}_{r}R_{0},...,
(\mathcal{P}^{-1}_{r}\mathcal{A})^{k-1}\mathcal{P}^{-1}_{r}R_{0}\}$ and $\mathrm{z}\in\mathbb{R}^{N}$ is the unknown solution.
\subsubsection{Properties of the preconditioner $\mathcal{P}_{r}$}
The distribution of eigenvalues and eigenvectors of a preconditioned matrix has a significant connection to how quickly Krylov subspace methods converge. Hence, it's valuable to analyze the spectral characteristics of the preconditioned matrix, denoted as  $\mathcal{P}_{r}^{-1}\mathcal{A}$. In the upcoming theorem, we will derive the pattern of eigenvalues for the matrix $\mathcal{P}_{r}^{-1}\mathcal{A}$.
\begin{proposition}
\cite{badahmane} The regularized preconditioner~(\ref{Pr}) has the block-triangular factorization:
\begin{eqnarray}
\left( \begin{array}{cc}
A\otimes I_3 & B^{T} \\ B & \beta Q	\end{array} \right)=\left( \begin{array}{cc}
I_{3n_{u}} & 0 \\ B(A^{-1}\otimes I_3)  & I_{n_p}	\end{array} \right)\left( \begin{array}{cc}
A\otimes I_{3} & 0\\ 0& \mathcal{S}	\end{array} \right)\left( \begin{array}{cc}
I_{3n_u} & (A^{-1}\otimes I_3)B \\ 0& I_{n_p}	\end{array} \right),
 \end{eqnarray}
where $\mathcal{S}=\beta Q-B(A^{-1}\otimes I_3)B^{T}$. If 
\( \beta \lambda_{\min}(Q) > \lambda_{\max}(B(A^{-1}\otimes I_3)B^T) \), 
then  \( \mathcal{S} \) is a positive definite matrix.
\end{proposition}

\begin{theorem}
\label{theo1}
	Let the preconditioner $\mathcal{P}_{r}$  be defined as in~(\ref{Pr}). Then for all $\beta>0$, $\mathcal{P}_{r}^{-1}\mathcal{A}$ has:
  \begin{itemize}
      \item  $1$ is an eigenvalue with multiplicity $n_u$, \item $\lambda_{1},..,\lambda_{n_{p}}$ eigenvalues,
\item If $\lambda_{max}(B^{T}(A^{-1}\otimes I_3)B)<\lambda_{min}(C)$,  then $\lambda>0$,
\item 
If $\lambda_{min}(B^{T}(A^{-1}\otimes I_3)B)>\lambda_{max}(C)$, then $\lambda<0$,
   \item  If $\displaystyle{\beta} \to -\displaystyle{\frac{c}{q}}$, then $$ \lambda \underset{\displaystyle{\beta}\to-\displaystyle{\frac{c}{q}}}
{\overset{}{\longrightarrow}}1,$$  
   \item   If $\beta \to 0$, then
$$ \lambda \underset{\beta \to 0}
{\overset{}{\longrightarrow}}1-\displaystyle{\frac{c}{a}},$$ 
\end{itemize}
\end{theorem}
\begin{proof}
Assume that $\lambda$ represents an eigenvalue of the preconditioned matrix  and $(\mathrm{u};\mathrm{p})$ is the associated eigenvector. In order to deduce the distribution of eigenvalues, we analyze the following generalized eigenvalue problem
	\begin{eqnarray}
	\label{P-1A}
 \mathcal{A} \left(\begin{array}{c}
	\mathrm{u}\\ 
    \mathrm{p}	\end{array}\right)=\lambda\mathcal{P}_{r}\left(\begin{array}{c}
	\mathrm{u}\\ 
    \mathrm{p}	\end{array}\right).
	\end{eqnarray}
	(\ref{P-1A})~can be reformulated  as follows
 
	\begin{equation}
	\label{RP1-A}
\left\{
\begin{array}{rl}
(1-\lambda)(A\otimes I_3)\mathrm{u}+(1-\lambda)B^{T}\mathrm{p}&=0, \\
 (1-\lambda)B\mathrm{u}+C\mathrm{p}&=\lambda\beta Q\mathrm{p}.
\end{array}
\right.
	\end{equation}
In the case where $\mathrm{p}=0$, equation~(\ref{RP1-A}) is always true for $\lambda=1$, consequently, there exist $n_{u}$ linearly independent eigenvectors $\left(\begin{array}{c}
     \mathrm{u}^{(i)}\\
     0
\end{array}\right)$, $i=1,..,n_u$, corresponding to the eigenvalue $1$, where $\mathrm{u}^{(i)}$  are arbitrary linearly independent vectors.
	If $\lambda\neq1$ and $\mathrm{u}=0$, from (\ref{RP1-A}), it can be deduced that $\mathrm{p}=0$ and $\mathrm{u}=0$. This conflicts with the initial assumption that the column vector $\left(\begin{array}{c}
     \mathrm{u} \\
     \mathrm{p}
\end{array}\right)$ is an eigenvector of the preconditioned matrix $\mathcal{P}_{r}^{-1}\mathcal{A}$.
If $\mathrm{p}=0$, then  from~(\ref{RP1-A}), we deduce that $\mathrm{u}$ must be $0$, this contradicts the initial assumption that $(\mathrm{u};\mathrm{p})$ is the eigenvector of $\mathcal{P}_{r}^{-1}\mathcal{A}$ and  therefore $\mathrm{u}\neq 0$ and  
 $\mathrm{p}\neq 0$, if $\mathrm{p}$ satisfies the second equation of~(\ref{RP1-A}), then 
 \begin{eqnarray}
 \label{Q1}
\left(B(A\otimes I_3)^{-1}B^{T}-C\right)\mathrm{p}&=&\lambda\left(B(A\otimes I_3)^{-1}B^{T}+\beta Q \right)\mathrm{p},
\end{eqnarray}
which can be rewritten equivalently as follows:
 \begin{eqnarray}
 \label{Q}
\left(B(A^{-1}\otimes I_3)B^{T}-C\right)\mathrm{p}&=&\lambda\left(B(A^{-1}\otimes I_3)B^{T}+\beta Q \right)\mathrm{p}, \nonumber
 \end{eqnarray}
 premultiplying  $(\ref{Q})$ with $\frac{\displaystyle{\mathrm{p}^{T}}}{\displaystyle{\mathrm{p}^{T}\mathrm{p}}}$ gives 
\begin{eqnarray}
	\label{lambda}
	\lambda=\frac{a-c}{a+\beta q},
	\end{eqnarray}
 where :
   \begin{itemize}
   \item  $a=\frac{\displaystyle{\mathrm{p}^{T}B^{T}(A^{-1}\otimes I_3)B\mathrm{p}}}{\displaystyle{\mathrm{p}^{T}\mathrm{p}}}$,
  \end{itemize}
since $A \succ 0$,  $Q\succ 0$ and $C\succeq 0$, $B$  have a full column rank, we know that  $a>0$.
It follows from $(\ref{lambda})$ that
	\begin{eqnarray}
	\label{lambda2}
	\lambda=1-\frac{\beta q+c}{a+\beta q},
	\end{eqnarray}
 \begin{itemize}
\item if $\lambda_{max}(B^{T}A^{-1}B)<\lambda_{min}(C)$,  then $\lambda>0$,
\item 
if $\lambda_{min}(B^{T}A^{-1}B)>\lambda_{max}(C)$, then $\lambda<0$,
   \item  if $\displaystyle{\beta} \to -\displaystyle{\frac{c}{q}}$, then $$ \lambda \underset{\displaystyle{\beta}\to-\displaystyle{\frac{c}{q}}}
{\overset{}{\longrightarrow}}1,$$  
   \item   if $\beta \to 0$, then
$$ \lambda \underset{\beta \to 0}
{\overset{}{\longrightarrow}}1-\displaystyle{\frac{c}{a}}.$$ 
\end{itemize}
Thus, the proof of theorem~\ref{theo1} is completed.
\end{proof}
\subsubsection{ Implementation of the regularized preconditioner $\mathcal{P}_{r}$~(\ref{Pr}).}
$(\ref{iteration})$ needs to be solved at each iteration step. To this end we use the following algorithm :
\begin{algorithm}[H]

\caption{: The regularized preconditioner $\mathcal{P}_{r}$}
\begin{algorithmic}
    \State 1: $\mathcal{P}_{r}\mathrm{z}=\mathcal{A}\mathrm{v}_{i}$, where $\mathrm{v}_{i}=(\mathrm{v}_{i}^{(1)};\mathrm{v}_{i}^{(2)})$,
       \State 2: $\mathcal{S}\mathrm{z}^{(2)}=\left(C- B(A^{-1}\otimes  I_{3})B^{T}\right)\mathrm{v}_{i}^{(2)}$,
    \textbf{ \State    3}: 
  $  (A\otimes  I_{3})\mathrm{z}^{(1)}=A\mathrm{v}_{i}^{(1)}+B^{T}\left(\mathrm{v}_{i}^{(2)}-\mathrm{z}^{(2)}\right),$   
    \State    6: $\mathrm{z}=\left(\begin{array}{c}
				\mathrm{z}^{(1)} \\ \mathrm{z}^{(2)}
				\end{array} \right).$
      \end{algorithmic} 
        \label{algor}
\end{algorithm}

\begin{itemize}
\item The matrix, unknown and right-hand side of linear system~\eqref{iteration} are first partitioned.
\item  
$\mathcal{S} \succ 0 $ and its sparsity pattern is more complicated than the sparsity pattern of $Q$ in general. We solve it iteratively  by $\mathcal{P}$CG method.
\item  
The step $3$, the matrix $A$ has a very high order. Since $A\succ 0 $, therefore, we can employ  $\mathcal{P}$CG to solve  the linear system. Additionally, the linear system with coefficient matrix $A$ can be solved  by  the Cholesky factorization in combination with approximate minimum degree (AMD) reordering from MATLAB library.
\item The solution $\mathrm{z}$ of block linear system~\eqref{iteration} is eventually obtained by concatenating partial solutions $\mathrm{z}^{(1)}$ and $\mathrm{z}^{(2)}$.
\end{itemize}
\begin{idea}
\label{idea}
In step $3$ of Algorithm \ref{algor},
 the primary motivation of this work is not to solve it  independently but rather, using  $\mathcal{P}$GCG method~\cite{GBICG} for  solving  linear system with multiple right-hand sides of the following  form :
\begin{eqnarray}
\label{PSMR}
A\mathcal{X}=\mathcal{H},
\end{eqnarray}
where: $\mathcal{X}$ and $\mathcal{H}$  are both  an $n\times3$ matrices.  Each column of matrix $\mathcal{X}$ is denoted as $\mathcal{X}^{(j)}$, and each column of matrix $\mathcal{H}$ is denoted as  $\mathcal{H}^{(j)}=A\mathrm{v}_{i}^{(j)}+B^{T}(\mathrm{v}_{i}^{(4)}-\mathrm{z}^{(4)})$, with $j=1,..,3$,
$\mathcal{X}_{0}$ is the initial guess of  solution (\ref{PSMR}) and $R_{0}=\mathcal{H}-A\mathcal{X}_{0}$ is the initial residual.
The preconditioning technique helps to improve the convergence rate of the $\mathrm{GCG}$ method by transforming the original linear system into an equivalent one with better properties, such as a better-conditioned matrix or a smaller size. 
\end{idea}
By leveraging the structure of the regularized preconditioner $\mathcal{P}_{r}$ (\ref{Pr}) and Idea $1$, in the rest of the article, we refer
to our new preconditioners as  $\mathcal{P}_{Gr}$, where
\begin{itemize}
    \item 
$\mathcal{P}_{Gr}$ : is the global regularized  preconditioner.
\end{itemize}
The following algorithmic version of the global regularized  preconditioner $\mathcal{P}_{Gr}$ for solving (\ref{iteration}),  can be
defined as follows:
\begin{algorithm}[H]

\caption{: The global regularized preconditioner $\mathcal{P}_{Gr}$}
\begin{algorithmic}
    \State 1: $\mathcal{P}_{Gr}\mathrm{z}=\mathcal{A}\mathrm{v}_{i}$, where 
    $\mathrm{v}_{i}=(\mathrm{v}_{i}^{(1)};\mathrm{v}_{i}^{(2)};\mathrm{v}_{i}^{(3)};\mathrm{v}_{i}^{(4)})$,
       \State 2: $\mathcal{S}\mathrm{p}=\left(C-B_{1}A^{-1}B_{1}^{T}-B_{2}A^{-1}B_{2}^{T}-B_{3}A^{-1}B_{3}^{T}\right)\mathrm{v}_{i}^{(4)}$,
    \textbf{ \State    3}:  $A\mathcal{X}=\mathcal{H},$
     \textbf{ \State    4}:  $\mathrm{z}=\left(\begin{array}{c}
				\mathcal{X}^{(1)} \\ \mathcal{X}^{(2)}\\ \mathcal{X}^{(3)}\\ 
                 \mathrm{p}
				\end{array} \right).$
      \end{algorithmic} 
        \label{algor2}
\end{algorithm}
\begin{itemize}
\item The matrix, unknown and right-hand side of linear system~\eqref{iteration} are first partitioned,
\item  
$\mathcal{S}\succ 0$, we solve it iteratively  by $\mathcal{P}$CG, where the preconditioner $\mathcal{P}$, utilized is based
on incomplete Cholesky factorization (ICHOL),
\item  
Since  $A\succ 0$, we  solve the linear system with multiple right-hand sides  iteratively by $\mathcal{P}$GCG  method \cite{GBICG},
\item The solution $\mathrm{z}$ of block linear system~\eqref{iteration} is  obtained  by combining partial solutions $\mathcal{X}^{(1)}$, $\mathcal{X}^{(2)}$, $\mathcal{X}^{(3)}$  and $\mathrm{p}$. 
\end{itemize}
\section{Global Krylov subspace methods}
In this section, we will review some properties and definitions of global methods~\cite{badahmane,GBICG}. To begin, we will introduce some symbols and notations that will be referenced throughout this paper.
\begin{definition}~(\cite{badahmane,GBICG}).
Let $\mathcal{X}$ and $\mathcal{Y}$ are two matrices of dimension $n_{u}\times 3$, we consider the scalar product 
\begin{eqnarray}
<\mathcal{X},\mathcal{Y}>_{F}=\mathrm{Trace}(\mathcal{X}^{T}\mathcal{Y}),
\end{eqnarray}
where $\text{Trace}$ is the trace of the square matrix $\mathcal{X}^{T}\mathcal{Y}$, refers to the sum of its diagonal entries.
The associated norm is the Frobenius norm, denoted $\|.\|_{F}$, which can be expressed as follows:
\begin{eqnarray}
 \label{normF}
\|\mathcal{X}\|_{F}=\sqrt{\mathrm{Trace}\left(\mathcal{X}^{T}\mathcal{X}\right)}.
\end{eqnarray}
\end{definition}
\begin{definition}
Let $V\in\mathbb{R}^{n_{u}\times 3}$, the global Krylov subspace $\mathcal{K}_{k}(A,V)$  is the
	subspace  spanned by the  matrices $V$, $AV$, ..., $A^{k-1}V.$
\end{definition}
\begin{remark}
\label{Kry}
Let $V$ be a real matrix of dimension $n_{u}\times 3$. According to the definition of the subspace $\mathcal{K}_k (A,V)$, we have 
\begin{eqnarray}
\mathcal{Z}\in \mathcal{K}_{k}(A,V)\iff \mathcal{Z}=\sum_{i=1}^{k}\psi_{i}A^{i-1}V, \hspace{1.2cm}
\psi_{i}\in\mathbb{R}, i=1,....,k.  \nonumber
\end{eqnarray}
In other words, $\mathcal{K}_k(A, V)$ is the subspace of $\mathbb{R}^{n_{u}\times 3}$ that
contains all the matrices of dimension $n_{u}\times 3$, written as  $\mathcal{Z}=\Theta(A)V$, where $\Theta$ is a polynomial of degree at most equal  $k-1$.
\end{remark}
\begin{definition}
	The Kronecker product of $M$ and $J$, where $M$ is a matrix with dimension $m \times r$ and $J$ is a matrix with dimensions $n \times s$, can be expressed in a specific mathematical form
	\begin{eqnarray}
 \label{Kronecker}
	M\otimes J=\left(\begin{array}{ccc}
	M_{11 }J &\hdots  &  M_{1r}J\\
	\vdots &\ddots &\vdots    \\
	M_{m1}J & \hdots &  M_{mr}J
	\end{array}\right)\in\mathbb{R}^{m n\times rs}.
   \end{eqnarray}
\end{definition}
\begin{definition}{(Diamond Product)}\label{Diamond}
	Let $Y$ and $\mathcal{Z}$ be a matrices of  dimensions $n_{u}\times 3s$ and $n_{u}\times 3r$, respectively. The matrices $Y$ and $\mathcal{Z}$ are constructed by concatenating $n_{u}\times 3$ matrices $Y_{i}$ and $\mathcal{Z}_{j}$ (for $i = 1,...., s$ and $j = 1,...,r)$. The symbol 
 $\diamond$ represents a specific product defined by the following formula:
	$$ Y^T
	\diamond \mathcal{Z}=
	\left(
	\begin{array}{cccc}
	\langle Y_{1},\mathcal{Z}_{1} \rangle_{F} & \langle Y_{1},\mathcal{Z}_{2} \rangle_{F} & \ldots & \langle Y_{1},\mathcal{Z}_{r} \rangle_{F} \\
	\langle Y_{2},\mathcal{Z}_{1} \rangle_{F}&\langle Y_{2},\mathcal{Z}_{2} \rangle_{F}&\ldots&\langle Y_{2},\mathcal{Z}_{r} \rangle_{F}\\
	\vdots&\vdots&\ddots&\vdots\\
	\langle Y_{s},\mathcal{Z}_{1} \rangle_{F}&\langle Y_{s},\mathcal{Z}_{2} \rangle_{F}&\ldots&\langle Y_{s},\mathcal{Z}_{r} \rangle_{F}
	\end{array}
	\right)\in\mathbb{R}^{s\times r}.$$
\end{definition}

\subsection{Convergence result for the preconditioned global conjugate gradient method ($\mathcal{P}$GCG)}
In this part of the paper, we will be referring to a mathematical operation called a  symmetric bilinear form on  $\mathbb{R}^{n_{u}\times 3}$
	\begin{eqnarray}
     \label{TraceP}
	&\mathbb{R}^{n_{u}\times 3}\times \mathbb{R}^{n_{u}\times 3}&\rightarrow \mathbb{R} \nonumber\\	&\left(\mathcal{X},\mathcal{Y}\right)&\rightarrow \langle \mathcal{X},\mathcal{Y}\rangle_{\mathcal{P}^{-1}}=\mathrm{Trace}\left(\mathcal{Y}^{T}\mathcal{P}^{-1}\mathcal{X}\right),
	\end{eqnarray}
 (\ref{TraceP}) can be expressed also  in terms of Diamond product defined in Definition~\ref{Diamond} 
\begin{eqnarray}
  \label{DiamondP}
	&\mathbb{R}^{n_{u}\times 3}\times \mathbb{R}^{n_{u}\times 3}&\rightarrow \mathbb{R} \nonumber\\
	&\left(\mathcal{X}, \mathcal{Y}\right)&\rightarrow \langle \mathcal{X},\mathcal{Y}\rangle_{\mathcal{P}^{-1}}=\mathcal{Y}^{T}\diamond \mathcal{P}^{-1}\mathcal{X},
	\end{eqnarray}
the associated  norm $\|.\|_{\mathcal{P}^{-1}}$ is given as
\begin{eqnarray}
\label{norm}
\|\mathcal{X}\|_{\mathcal{P}^{-1}}=\sqrt{\mathrm{Trace}\left(\mathcal{X}^{T}\mathcal{P}^{-1}\mathcal{X}\right)},  \hspace{0.1cm}\text{for}\hspace{0.1cm}  \mathcal{X}\in\mathbb{R}^{n_{u}\times 3}.
\end{eqnarray}
The preconditioned  global conjugate gradient method is a mathematical algorithm that can be used to solve a linear system with multiple right-hand sides~(\ref{PSMR}). It is based on the fact that the preconditioning matrix $\mathcal{P}$ can be decomposed using the ICHOL decomposition into a lower triangular matrix $G$ and its transpose $G^{T}$. This decomposition can be used to apply a central preconditioning technique that preserves the symmetry of the matrix $A$, which can lead to faster and more accurate solutions
\begin{eqnarray}
\label{PSMR1}
G^{-1}AG^{-T}\mathcal{Y}=G^{-1}\mathcal{H} \hspace{ 0.2cm} \text{with}\hspace{ 0.2cm}\mathcal{Y}=G^{T}\mathcal{X}.
\end{eqnarray}
The  resulting matrix $G^{-1}AG^{-T}$ will still be  symmetric positive definite matrix. Additionally  the initial  residual  of~(\ref{PSMR1}) is equal to  $G^{-1}R_{0}$, which $R_{0}$ is the initial residual of~(\ref{PSMR}).
 \begin{theorem}
 \label{PGCG}
The $\mathcal{P}$GCG method constructs at step $k$, the approximation
$\mathcal{Y}_k$ satisfying the following two relations :
\begin{equation}\label{gmres0}
\mathcal{Y}_k-\mathcal{Y}_0 \in \mathcal{K}_{k}(A\mathcal{P}^{-1},R_{0})\; {\rm and} \; {R}_k{\perp}_{{\langle .,.\rangle}_{\mathcal{P}^{-1}}} \mathcal{K}_{k}
(A\mathcal{P}^{-1},R_{0}).
\end{equation}
\end{theorem}
In the following, we will establish expressions for the $\mathcal{P}^{-1}A\mathcal{P}^{-1}$-norm of the error, represented by $\|\mathcal{E}_{k}\|_{\mathcal{P}^{-1}A\mathcal{P}^{-1}}$, where $\mathcal{E}_{k}=\mathcal{Y}^{*}-\mathcal{Y}_k$  , a similar findings for the error norm of other Krylov methods, such as  standard CG or  global FOM, have been proven in a  previous study \cite{badahmane1,badahmane}. Firstly, we will provide formulation for $\|\mathcal{E}_{k}\|_{\mathcal{P}^{-1}A\mathcal{P}^{-1}}$ that involve the global Krylov matrix $K_{k}=\left[R_0,\left(A\mathcal{P}^{-1}\right)R_0,...,\left(A\mathcal{P}^{-1}\right)^{k-1} R_0\right]$.
\begin{theorem}
If the matrices ${K}^{T}_{k}\diamond \mathcal{P}^{-1}A\mathcal{P}^{-1}{K}_k$  and ${K}^{T}_{k+1}\diamond A^{-1}{K}_{k+1}$   are  nonsingular, then we can express the $\mathcal{P}$GCG error  at iteration $k$ as  follows 
\begin{eqnarray}
\label{errorNorm}
\|\mathcal{E}_{k}\|_{\mathcal{P}^{-1}A\mathcal{P}^{-1}}^{2}&=&\frac{1}{\displaystyle{e_1^T\left(K_{k+1}^{T}\diamond A^{-1}K_{k+1}\right)^{-1} e_1}},
\label{errNormdet}
\label{residueNrom}
\end{eqnarray}
where 
$K_{k+1}=\left[R_0,\left(A\mathcal{P}^{-1}\right)R_0,...,\left(A\mathcal{P}^{-1}\right)^{k} R_0\right]$  and  $e_1$   is the first unit vector of  $\mathbb{R}^{k+1}$.
\end{theorem}
\begin{proof}
Using Theorem~\ref{PGCG}, we can express the error vector $\mathcal{E}_k$ as a combination of global Krylov matrices, specifically the matrices  $\left(A\mathcal{P}^{-1}\right)^{i-1}R_0$, $i=1,...,k$. This can be written as a follows :
\begin{eqnarray}
\label{errorV}
\mathcal{E}_k=\mathcal{Y}^{*}-\mathcal{Y}_{k},
\end{eqnarray}
where $\mathcal{Y}^{*}$ is the exact solution of~(\ref{PSMR}). (\ref{errorV}) can be rewritten as follows
\begin{eqnarray}
\label{errorVrr2} 
\mathcal{E}_k&=&\mathcal{Y}^{*}-\mathcal{Y}_{k}, \nonumber \\
  &=&\mathcal{Y}^{*}-\mathcal{Y}_{0}+ \mathcal{Y}_{0}-\mathcal{Y}_{k},  \nonumber \\
  &=& \mathcal{E}_{0}-\left(\mathcal{Y}_{k}-\mathcal{Y}_{0}\right). 
\end{eqnarray}
 \text{By using Theorem}~\ref{PGCG}, we get that 
 $\left(\mathcal{Y}_{k}-\mathcal{Y}_{0}\right)\in \mathcal{K}_{k}
(A\mathcal{P}^{-1},R_{0})$, additionally, based on  Remark $3.1$, we obtain that 
\begin{eqnarray}
\label{errorVrrr1} 
  \mathcal{Y}_{k}- \mathcal{Y}_{0}&=& \sum_{i=1}^{k}\psi_{i}\left(A\mathcal{P}^{-1}\right)^{i-1}R_{0},\hspace{0.2cm} where \hspace{0.2cm} (\psi_{1},...,\psi_{k})\in\mathbb{R}^{k}.
\end{eqnarray}
Substituting equation (\ref{errorVrrr1}) into (\ref{errorVrr2}), results the following expression
\begin{eqnarray}
\label{errorVrrr} 
  \mathcal{E}_{k}&=& \mathcal{E}_{0}-\sum_{i=1}^{k}\psi_{i}\left(A\mathcal{P}^{-1}\right)^{i-1}R_{0},\nonumber
\end{eqnarray}
as we know that  $R_{k}=A\mathcal{P}^{-1}\mathcal{E}_{k}$, then we can express the residual matrix $R_{k}$  as follows
\begin{eqnarray}
\label{residue}
    R_{k}= R_{0}-A\mathcal{P}^{-1}\sum_{i=1}^{k}\psi_{i}\left(A\mathcal{P}^{-1}\right)^{i-1}R_{0},
\end{eqnarray}
by using  orthogonality condition in Theorem~\ref{PGCG}, we obtain 
\begin{eqnarray}
\label{Ortho}
\langle R_{k},\left(A\mathcal{P}^{-1}\right)^{i-1} R_{0}\rangle_{\mathcal{P}^{-1}}&=&0,\nonumber
\end{eqnarray}
using  (\ref{residue}), we obtain the following equality 
\begin{eqnarray}
\langle R_{0},\left(A\mathcal{P}^{-1}\right)^{i-1} R_{0}\rangle_{\mathcal{P}^{-1}}&=&\langle A\mathcal{P}^{-1}\sum_{i=1}^{k}\psi_{i}\left(A\mathcal{P}^{-1}\right)^{i-1}R_{0},\left(A\mathcal{P}^{-1}\right)^{i-1} R_{0}\rangle_{\mathcal{P}^{-1}}.\nonumber
\end{eqnarray}
By utilizing  the definition of  ${\langle . , .\rangle}_{\mathcal{P}^{-1}}$  and employing   the expression of the  matrix  $K_{k}$, we obtain  the following expression
\begin{eqnarray}
K_{k}^{T}\diamond\mathcal{P}^{-1}R_{0}&=&K_{k}^{T}\diamond\mathcal{P}^{-1}A\mathcal{P}^{-1}K_{k}\left(\psi\otimes I_{3}\right),  \text{where}\hspace{0.1cm} \psi=\left(\psi_{1},..,\psi_{k}\right),  \nonumber \\
&=& \left(K_{k}^{T}\diamond\mathcal{P}^{-1}A\mathcal{P}^{-1}K_{k}\right)\psi.
\end{eqnarray}
If the matrix    
$K_{k}^{T}\diamond\mathcal{P}^{-1}A\mathcal{P}^{-1}K_{k}$ is nonsingular, then 
\begin{eqnarray}
\label{psi}
\psi=\left(K_{k}^{T}\diamond\mathcal{P}^{-1}A\mathcal{P}^{-1}K_{k}\right)^{-1}K_{k}^{T}\diamond\mathcal{P}^{-1}R_{0}.
\end{eqnarray}
Due to the orthogonality conditions, it can be deduced from Theorem~\ref{PGCG} and relation (\ref{residue}), given that
\begin{eqnarray}
\langle \left(A\mathcal{P}^{-1}\right)\mathcal{E}_{k},\mathcal{E}_{k}\rangle_{\mathcal{P}^{-1}}&=&
\langle {R}_{k},\mathcal{E}_{0}\rangle_{\mathcal{P}^{-1}}.\nonumber 
\end{eqnarray}
Using ${\langle . , .\rangle}_{\mathcal{P}^{-1}}$(\ref{DiamondP})
and  (\ref{residue}), we deduce that 
\begin{eqnarray}
\label{epsilon}
\|\mathcal{E}_{k}\|_{\mathcal{P}^{-1}A\mathcal{P}^{-1}}^{2}&=& \mathcal{E}_{0}^{T}\diamond\mathcal{P}^{-1}R_{0}-\left(\mathcal{E}_{0}^{T}\diamond\mathcal{P}^{-1}A\mathcal{P}^{-1}K_{k}\right)\psi,
\end{eqnarray}
by substituting the expression of $\psi$ $(\ref{psi})$ in the right-hand side of $(\ref{epsilon})$, we obtain the following expression:
\begin{eqnarray}
\label{epsilon1}
\|\mathcal{E}_{k}\|_{\mathcal{P}^{-1}A\mathcal{P}^{-1}}^{2}=\mathcal{E}_{0}^{T}\diamond\mathcal{P}^{-1}R_{0}-\left(\mathcal{E}_{0}^{T}\diamond\mathcal{P}^{-1}A\mathcal{P}^{-1}K_{k}\right)\left(K_{k}^{T}\diamond\mathcal{P}^{-1}A\mathcal{P}^{-1}K_{k}\right)^{-1}\left(K_{k}^{T}\diamond\mathcal{P}^{-1}R_{0} \right),
\end{eqnarray}
we observe that the right-hand side of (\ref{epsilon1})  is a Schur complement for the matrix
\begin{eqnarray}
\label{Schur}
K_{k+1}^{T}\diamond A^{-1}K_{k+1}=\left(\begin{array}{cc}
\mathcal{E}_{0}^{T}\diamond\mathcal{P}^{-1}R_{0}& \mathcal{E}_{0}^{T}\diamond\mathcal{P}^{-1}A\mathcal{P}^{-1}K_{k}\\ K_{k}^{T}\diamond\mathcal{P}^{-1}A\mathcal{P}^{-1}\mathcal{E}_{0}&
K_{k}^{T}\diamond\mathcal{P}^{-1}A\mathcal{P}^{-1}K_{k}
\end{array} \right).
 \end{eqnarray}
The  matrix (\ref{Schur}) has the following block-triangular factorization
 \begin{eqnarray}
 \label{decom}
\left[ \begin{array}{cc}
1& \left(\mathcal{E}_{0}^{T}\diamond\mathcal{P}^{-1}A\mathcal{P}^{-1}K_{k}\right)(K_{k}^{T}\diamond\mathcal{P}^{-1}A\mathcal{P}^{-1}K_{k})^{-1}\\ 0&
I
\end{array} \right]\left[\begin{array}{cc}
\|\mathcal{E}_{k}\|_{\mathcal{P}^{-1}A\mathcal{P}^{-1}}^{2}& 0\\ K_{k}^{T}\diamond\mathcal{P}^{-1}A\mathcal{P}^{-1}\mathcal{E}_{0}&
K_{k}^{T}\diamond\mathcal{P}^{-1}A\mathcal{P}^{-1}K_{k}
\end{array}\right],
 \end{eqnarray}
 using the hypothesis that $K_{k+1}^{T}\diamond A^{-1}K_{k+1}$  and $K_{k}^{T}\diamond\mathcal{P}^{-1}A\mathcal{P}^{-1}K_{k}$ are nonsingular and by using the inverse of the decomposition (\ref{decom}), the following result
holds:
\begin{eqnarray}
\label{errorNorm3}
\|\mathcal{E}_{k}\|_{\mathcal{P}^{-1}A\mathcal{P}^{-1}}^{2}&=&\frac{1}{\displaystyle{e_1^T\left(K_{k+1}^{T}\diamond A^{-1}K_{k+1}\right)^{-1} e_1}}.
\end{eqnarray}
\end{proof}
The matrix  $A$ is symmetric positive definite and  can be decomposed using the Cholesky decomposition as  $A=LL^{T}$. When considering the expression $L^{-1}(A\mathcal{P}^{-1})L$, which involves the inverse of $L$, we observe that this expression is symmetric. Consequently, the spectral decomposition of $L^{-1}(A\mathcal{P}^{-1})L$ can be determined as follows
$L^{-1}\left(A\mathcal{P}^{-1}\right)L=\mathcal{V}\Lambda\mathcal{V}^{T}$, where $\Lambda$ is diagonal matrix whose elements are the eigenvalues $\lambda_{1},..,\lambda_{n_u}$ and $\mathcal{V}$ is the eigenvector matrix. We will give a more sharp upper bound for the residual and error norms.
\begin{theorem}
Let the initial residual  $R_0$   be decomposed as $R_0=L\mathcal{V} \gamma$
	where $\gamma^{(j)}$ is  vector with components $\gamma_{1}^{(j)}$,.., $\gamma_{n_u}^{(j)}$, with $j=1,..,3$.
 In this case, we can express it as follows:
	\begin{eqnarray}\label{conv1}
	&(1)&\hspace{0.1cm}\|\mathcal{E}_k\|_{\mathcal{P}^{-1}A\mathcal{P}^{-1}}^{2} =\displaystyle \frac{1}{e_1^T(V_{k+1}^T
		\widetilde D V_{k+1})^{-1} e_1},\nonumber\\
	&(2)&\hspace{0.1cm}\|R_k\|_{\mathcal{P}^{-1}}^{2} \leq\displaystyle \frac{1}{e_1^T(V_{k+1}^T
		\widetilde D V_{k+1})^{-1} e_1}\left(\frac{\tilde{\theta}^{k+1}}{\theta^{k}}\right),\nonumber
	\end{eqnarray}
	where
	\begin{equation}\label{krylov3}
\widetilde{D}=\text{diag}\left(\sum_{j=1}^{3}|\gamma_{1}^{(j)}|^{2},..., \sum_{j=1}^{3}|\gamma_{n_{u}}^{(j)}|^{2}\right)
	\quad {\rm and} \quad
	V_{k+1}=\left(
	\begin{array}{cccc}
	1&\lambda_1&\ldots&\lambda_1^{k}\\
	\vdots&\vdots&&\vdots\\
	1&\lambda_{n_u}&\ldots&\lambda_{n_u}^{k}\\
	\end{array}
	\right),
	\end{equation}
	$e_1$ is the first unit
	vector of $\mathbb{R}^{k+1}$,  $\widetilde{\theta}$ and $\theta$ are the maximum and the minimum eigenvalues  respectively of the tridiagonal  matrices  $\hat{T}_{k+1}=\left(U_{k+1}^{T}\diamond A^{-1}U_{k+1}\right)^{-1}$ and $T_{k}=U_{k}^{T}\diamond\mathcal{P}^{-1}A\mathcal{P}^{-1}U_{k}$, where $U_{k}^{T}\diamond \mathcal{P}^{-1}U_{k}=I_{k}$,
$$
 T_{k}=\begin{bmatrix}
   \alpha_{1} & \eta_{2}  \\ 
  \eta_{2} & \alpha_{2} & \eta_{3}  \\ 
   &  \ddots&  \ddots & \ddots\\ 
   &  & \eta_{k-1}  & \alpha_{k-1}  &\eta_{k}
                       \\
     &  &   & \eta_{k}  &\alpha_{k}
 \end{bmatrix},
$$
and 
$$
\hat{T}_{k+1}=
 \begin{bmatrix}
   \alpha_{1} & \eta_{2}  \\ 
  \eta_{2}& {\alpha}_{2} & {\eta}_{3}  \\ 
   &  \ddots&  \ddots & \ddots\\ 
   &  & {\eta}_{k}  & {\alpha}_{k}  &{\eta}_{k+1}
                       \\
     &  &   & {\eta}_{k+1}  &\hat{\eta}_{k+1}\\
 \end{bmatrix},
$$
 obtained by Lanczos.
\end{theorem}
\begin{proof}
	$L^{-1}K_{j,k+1}$ can be written as
	\begin{equation}\label{krylov1}
	L^{-1}K_{j,k+1}=[\mathcal{V}\gamma^{(j)},\mathcal{V}\,\Lambda\gamma^{(j)},\ldots,\mathcal{V}\Lambda^{k}\gamma^{(j)}],\quad   \nonumber
	\end{equation}
	\noindent which is equivalent to
	\begin{equation}\label{krylov2}
L^{-1}K_{j,k+1}=\mathcal{V}D_{\gamma^{(j)}}V_{k+1},
	\end{equation}
	where $D_{\gamma^{(j)}}$ 
	\begin{equation}
	D_{\gamma^{(j)}}=\left(
	\begin{array}{cccc}
	{\gamma_1^{(j)}}&&&\\
	&\ddots&&\\
	&&&{\gamma_{n_u}^{(j)}}
	\end{array}
	\right). \nonumber
	\end{equation}
	From (\ref{krylov2}), it
	follows that
\begin{equation}
\label{krylov4}
(L^{-1}K_{k+1})^T\diamond(L^{-1}K_{k+1})=\sum_{j=1}^{3}(L^{-1}K_{j,k+1})^T(L^{-1}K_{j,k+1})=V_{k+1}^T\sum_{j=1}^{3}\,D_{\gamma^{(j)}}^T\,\,D_{\gamma^{(j)}}V_{k+1}.
\end{equation}
	Using  (\ref{errorNorm}), (\ref{krylov4}) and  the fact that $\widetilde D=\sum_{j=1}^{3}
	D_{\gamma^{(j)}}^T\,D_{\gamma^{(j)}}$, we obtain the following equality 
	\begin{equation}
	\label{2l}
	\|\mathcal{E}_k\|_{\mathcal{P}^{-1}A\mathcal{P}^{-1}}^{2} =\displaystyle \frac{1}{e_1^T(V_{k+1}^T
		\widetilde D V_{k+1})^{-1} e_1}.
	\end{equation}
Now, we will proceed to prove statement $(2)$. Utilizing equation  (\ref{residueNrom}) and the representation of the residue as a function of the determinant in \cite{badahmane1}, we can establish the following equality:
	\begin{eqnarray}
	\label{2zw}
\|R_{k}\|_{{\mathcal{P}^{-1}}}^{2}&=&\displaystyle \|\mathcal{E}_{k}\|_{\mathcal{P}^{-1}A\mathcal{P}^{-1}}^{2}\frac{\displaystyle{\text{det}\left(K_{k}^{T}\diamond\mathcal{P}^{-1}K_{k}\right)\text{det}\left(K_{k+1}^{T}\diamond\mathcal{P}^{-1}K_{k+1}\right)}}{\displaystyle{\text{det}\left(K_{k}^{T}\diamond\mathcal{P}^{-1}A\mathcal{P}^{-1}K_{k}\right)}\text{det}\left(K_{k+1}^{T}\diamond A^{-1}K_{k+1}\right)},
	\end{eqnarray}
	by substituting the QR factorization  of $K_k$ as $K_k=U_{k}{E}_{k}$ in  (\ref{2zw}), we derive the following  result:
	\begin{eqnarray}	
 \|R_{k}\|_{\mathcal{P}^{-1}}^{2} &=&\displaystyle \|\mathcal{E}_{k}\|_{\mathcal{P}^{-1}A\mathcal{P}^{-1}}^{2} \frac{\text{det}(\hat{T}_{k+1})}{\text{det}(T_{k})},\nonumber
	\end{eqnarray}
	where  $T_{k}=U_{k}^{T}\diamond\mathcal{P}^{-1}A\mathcal{P}^{-1}U_{k}$ and $\hat{T}_{k+1}=(U_{k+1}^{T}\diamond A^{-1}U_{k+1})^{-1}$.
	From (\ref{2l}), we obtain 
	\begin{eqnarray}
	\label{3t}
	\|R_{k}\|_{\mathcal{P}^{-1}}^{2} &=&\displaystyle \frac{1}{e_1^T(V_{k+1}^T
		\widetilde D V_{k+1})^{-1} e_1} \frac{\text{det}(\hat{T}_{k+1})}{\text{det}(T_{k})},
	\end{eqnarray}
	 it follows from  $(\ref{3t})$ that 
	\begin{equation}
	\|R_k\|_{\mathcal{P}^{-1}}^{2} \leq \displaystyle \frac{1}{e_1^T(V_{k+1}^T
		\widetilde D V_{k+1})^{-1} e_1}\left(\frac{\tilde{\theta}^{k+1}}{\theta^{k}}\right),
	\end{equation}
	where $\widetilde\theta$ and $\theta$ are the maximal and the minimal eigenvalues  respectively of  $\hat{T}_{k+1}$ and ${T}_{k}$.
\end{proof}

\section{Numerical results}\label{Numerical}
In this section, we present the outcomes of numerical experiments demonstrating the convergence behavior of the $\mathcal{P}_{Gr}$GMRES and $\mathcal{P}_{Gr}\mathcal{F}$GMRES  methods.  All the computations  were performed on a computer equipped with an a 64-bit 2.49 GHz core i5 processor and 8.00 GB RAM
using MATLAB.R2017.a. The initial iteration ${x}_{0}$  is set to be a zero vector.
The parameters $\beta$ is  chosen for testing the $\mathcal{P}_{r}$ and  $\mathcal{P}_{Gr}$ preconditioners, is set to   $\beta^{*}$  as found in previous studies~\cite{badahmane1}.
In  the numerical experiments: 
\begin{itemize}
    \item  $\mathrm{IT}$: number of iterations of the $\mathcal{P}\mathcal{R}$GMRES  or $\mathcal{F}$GMRES methods,
    \item  $\mathrm{RES}$: norm of absolute residual vector and  defined as follows
    \begin{eqnarray}
    \mathrm{RES}=\|d-\mathcal{A}{x}_{k}\|_{2},  \nonumber
    \end{eqnarray}
    where $x_{k}$ is the computed kth approximate solution.
    \item $\mathrm{RRES}$ denote the norm of absolute relative residual is defined as :
       \begin{eqnarray}
    \mathrm{RRES} =\frac{\|d-\mathcal{A}{x}_{k}\|_{2}}{\|d\|_{2}}.\nonumber
     \end{eqnarray}
\end{itemize}
\begin{exe}
\label{ex3}
We  consider a test problem derived from a benchmark problem described in \cite{koko}. The inflow and outflow conditions  are
\[
\begin{aligned}
u_1&= 0.3 \times 0.412 \times 4y(0.41 - y), \quad u_2 = 0 \quad \text{on} \quad \Gamma_{\text{in}} = \{0\} \times (0, 0.41), \\
u_1&= 0.3 \times 0.412 \times 4y(0.41 - y), \quad u_2 = 0 \quad \text{on} \quad \Gamma_{\text{out}} = \{2.2\} \times (0, 0.41).
\end{aligned}
\]
On the other parts of the boundary of \(\Omega\), homogeneous boundary conditions are prescribed (i.e., \(u = 0\)). The domain is discretized by $P_{1}-$Bubble$/P_1$, where
\begin{itemize}
    \item  $P_{1}-$Bubble: Lagrange element that is enhanced by the inclusion of a cubic bubble function,
    \item $\mathrm{P}_{1}$: Lagrange element  for the pressure,
\end{itemize}
the nodal positions of this mixed finite element  are illustrated in the following Fig.~\ref{nodal}:
\begin{figure}[H]
\tikzstyle{quadri}=[circle,draw,fill=black,text=white]
\tikzstyle{quadri2}=[circle,draw,fill=white,text=black]
\begin{center}
\begin{tikzpicture}
	\draw[->] (0,0) -- (3,0) node[right] {$\mathcal{L}_{1}$};
	\draw[->] (0,0) -- (0,3) node[above] {$\mathcal{L}_{2}$};
\end{tikzpicture}
\hspace{1cm}
\begin{tikzpicture}
\node[quadri] (A) at(-2,4) {};
\node[quadri] (W) at(-4,1){};
\node[quadri] (O) at(0,1){};
\node[quadri] (I) at(-2,2){};
\draw[-,=latex] (W)--(O);
\draw[-,=latex] (O)--(A);
\draw[-,=latex] (W)--(A);
\end{tikzpicture}
\end{center}
\tikzstyle{quadri2}=[circle,draw,fill=white,text=white]
\begin{center}
\begin{tikzpicture}
	\draw[->] (0,0) -- (3,0) node[right] {$\mathcal{L}_{1}$};
	\draw[->] (0,0) -- (0,3) node[above] {$\mathcal{L}_{2}$};
\end{tikzpicture}
\hspace{1cm}
\begin{tikzpicture}
\node[quadri2] (A) at(-2,4) {};
\node[quadri2] (W) at(-4,1){};
\node[quadri2] (O) at(0,1){};
\draw[-,=latex] (W)--(O);
\draw[-,=latex] (O)--(A);
\draw[-,=latex] (W)--(A);
\end{tikzpicture}
\end{center}
\caption{}\text{$\mathrm{P}_{1}-Bubble/\mathrm{P}_{1}$ element  $\left(\begin{tikzpicture}\node[quadri] (P) at (0,0)  {};\end{tikzpicture}\hspace{0.1cm}\text{ velocity node}; \begin{tikzpicture}\node[quadri2] (Q) at (0,0) {};\end{tikzpicture}\hspace{0.1cm} \text{pressure node} \right),\hspace{0.1cm}\text{local co-ordinate}  \left(\mathcal{L}_{1},\mathcal{L}_{2}\right)$.}
\label{nodal}
\end{figure}
then we obtain the  nonsingular saddle-point problem $(\ref{saddle})$, for more details, we refer to see~\cite{koko}. 
\end{exe}
In Tables $\ref{tab:v=1,exp3}$, $\ref{tab:v=0.1,exp3}$ and $\ref{tab:v=0.01,exp3}$ we report the total  required number of outer $\mathcal{R}$GMRES iterations, elapsed $\mathrm{CPU}$ time (in seconds),  residual  and relative residual
under "$\mathrm{IT}$" , "$\mathrm{CPU}$", "$\mathrm{RES}$" and 
"$\mathrm{RRES}$" with respect to different values $\nu$. In the following numerical results we will increase Reynolds numbers by varying $\nu$, specifically for
$(\nu = 0.001, \nu =0.01\hspace{0.1cm} \text{and}\hspace{0.1cm}\nu=0.1\hspace{0.1cm})$.\\
In the case $\nu =0.001.$\\
\begin{table}[H]
\centering
\caption{{Numerical results of  the $\mathcal{P}\mathcal{R}$GMRES methods.}}
		\begin{tabular}{ |p{1.4cm}||p{1.4cm}||p{1.8cm}||p{1.8cm}|}
			\hline
			$\alpha$ & & $\mathcal{P}_{r}\left(\beta^{*}\right)$ &  $\mathcal{P}_{Gr}\left(\beta^{*}\right)$\\
			\hline
			$1e+02$ &IT CPU RES RRES&  81\hspace{1.6cm} 0.74\hspace{1.6cm} 1.00e-02 \hspace{1.5cm} 1.50e-06&  56\hspace{1.6cm}
0.33\hspace{1.6cm} 2.00e-03\hspace{0.2cm}    3.99e-07
			\\
			\hline 
			$1e+03$&IT CPU RES RRES  & 111\hspace{1.8cm} 0.50\hspace{1.8cm}5.06e-02\hspace{0.8cm}7.10e-07
 &66 \hspace{1.8cm}0.38\hspace{0.8cm}   5.5e-03\hspace{1.8cm}   7.78e-08
			\\
			\hline
	$1e+04$&IT CPU RES RRES  & 141\hspace{1.8cm} 0.70\hspace{0.8cm}4.15e-01\hspace{0.8cm} 2.77e-06
 &69 \hspace{2.9cm}0.40\hspace{1.8cm} 3.00e-03\hspace{0.8cm}   4.29e-08
\\
\hline
		\end{tabular}
\label{tab:v=1,exp3}
\end{table}\vspace{0.1cm}
In the case $\nu =0.01.$
\begin{table}[H]
\centering
\caption{{Numerical results of  the $\mathcal{P}\mathcal{R}$GMRES methods.}}
		\begin{tabular}{ |p{1.4cm}||p{1.4cm}||p{1.8cm}||p{1.8cm}|}
			\hline
			$\alpha$ & & $\mathcal{P}_{r}\left(\beta^{*}\right)$ &  $\mathcal{P}_{Gr}\left(\beta^{*}\right)$\\
			\hline
			$1e+02$ &IT CPU RES RRES&  116\hspace{1.6cm} 1.19\hspace{1.6cm} 1.77e-05\hspace{1.5cm} 2.47e-09&  62\hspace{1.6cm}
0.73\hspace{1.6cm}  4.08e-07\hspace{0.2cm}  5.71e-11   
			\\
			\hline 
			$1e+03$&IT CPU RES RRES  & 110\hspace{1.8cm} 0.84\hspace{1.8cm} 
       5.3e-03\hspace{0.8cm}7.50e-08
 &67 \hspace{1.8cm}0.66\hspace{0.8cm}    1.88e-05\hspace{1.8cm}   2.64e-10
			\\
			\hline
	$1e+04$&IT CPU RES RRES  & 147\hspace{1.8cm} 1.45\hspace{0.8cm} 3.8e-02 \hspace{0.8cm}  5.37e-08
 &80\hspace{2.9cm}1.09\hspace{1.8cm} 3.02e-04\hspace{0.8cm}   4.24e-10
\\
\hline
		\end{tabular}
\label{tab:v=0.1,exp3}
\end{table}
In the case $\nu =0.1.$\\
\begin{table}[H]
\centering
\caption{{Numerical results of  the $\mathcal{P}\mathcal{R}$GMRES methods.}}
		\begin{tabular}{ |p{1.4cm}||p{1.4cm}||p{1.8cm}||p{1.8cm}|}
			\hline
			$\alpha$ & & $\mathcal{P}_{r}\left(\beta^{*}\right)$ &  $\mathcal{P}_{Gr}\left(\beta^{*}\right)$\\
			\hline
			$1e+02$ &IT CPU RES RRES&  123\hspace{1.6cm} 1.07\hspace{1.6cm}1.72e-05
\hspace{1.5cm} 2.41e-09&  67\hspace{1.6cm}
0.78\hspace{1.6cm}  2.17e-06\hspace{0.2cm}  3.04e-10 
			\\
			\hline 
			$1e+03$&IT CPU RES RRES  & 134\hspace{1.8cm} 1.33\hspace{1.8cm} 
       4.6e-03\hspace{0.8cm} 6.44e-08
 &83\hspace{1.8cm}0.98\hspace{0.8cm}    
   2.44e-05\hspace{1.8cm}  3.43e-10
			\\
			\hline
	$1e+04$&IT CPU RES RRES  & 162\hspace{1.8cm} 1.69\hspace{0.8cm} 1.30e-02 \hspace{0.8cm}  1.83e-08
 &99\hspace{2.9cm}1.39\hspace{1.8cm}  3.54e-05\hspace{0.8cm}   4.98e-11
\\ 
\hline
		\end{tabular}
\label{tab:v=0.01,exp3}
\end{table}
According to Tables~\ref{tab:v=1,exp3}, \ref{tab:v=0.1,exp3} and \ref{tab:v=0.01,exp3}, it can be observed that the global regularized preconditioner $\mathcal{P}_{Gr}$
 requires lowest $\mathrm{IT}$ and CPU time, which implies that the  $\mathcal{P}_{Gr}$GMRES method is superior to the  $\mathcal{P}_{r}$GMRES
 method in terms of computing efficiency. Using Algorithm $2$, which employs the $\mathcal{P}$GCG method to solve linear systems with multiple right-hand sides $(\ref{PSMR})$, lead to much better numerical results compared to solve it independently as described in  Algorithm \ref{algor}. This is evident as the  $\mathcal{P}_{Gr}\mathcal{R}$GMRES requires less CPU time for different values of parameter $\alpha$. The   $\mathcal{P}_{r}$ preconditioner  performs less effectively, particularly when the  parameter $\alpha$ increase.\\
 We list the numerical results of the  $\mathcal{F}$GMRES methods with  various values of $\nu$ in Tables $4$, $5$ and $6$.\\
In the case $\nu =0.001.$\\
\begin{table}[H]
\centering
\caption{{Numerical results of  the $\mathcal{F}$GMRES methods.}}
		\begin{tabular}{ |p{1.4cm}||p{1.0cm}|||p{1.8cm}||p{1.8cm}|}
			\hline
			$\alpha$ & & $\mathcal{P}_{r}\left(\beta^{*}\right)$ &  $\mathcal{P}_{Gr}\left(\beta^{*}\right)$\\
			\hline
			$1e+02$ &IT CPU RES RRES& 
			91\hspace{1.6cm} 1.03\hspace{1.6cm} 6.00e-03 \hspace{1.5cm}   8.41e-07&  55\hspace{1.6cm}
0.78\hspace{1.6cm} 6.00e-03\hspace{0.2cm}  9.91e-07
			\\
			\hline 
			$1e+03$&IT CPU RES RRES   & 77\hspace{1.8cm} 0.75 \hspace{1.8cm} 6.63e-02\hspace{0.8cm}9.30e-07
 & 49 \hspace{1.8cm}0.49\hspace{0.8cm} 6.70e-02\hspace{1.8cm}    9.41e-07
			\\
			\hline
	$1e+04$&IT CPU RES RRES   & 70\hspace{1.8cm}0.67\hspace{1.8cm}  6.6e-01\hspace{0.8cm} 9.26e-07 
 &43 \hspace{2.9cm}0.54\hspace{1.8cm}  6.6e-01\hspace{0.8cm} 8.57e-07
\\
\hline
		\end{tabular}
\label{tab:v=1,exp4}
\end{table}

In the case $\nu =0.01.$\\
\begin{table}[H]
\centering
\caption{{Numerical results of  the $\mathcal{F}$GMRES methods.}}
		\begin{tabular}{ |p{1.4cm}||p{1.0cm}|||p{1.8cm}||p{1.8cm}|}
			\hline
			$\alpha$ & & $\mathcal{P}_{r}\left(\beta^{*}\right)$ &  $\mathcal{P}_{Gr}\left(\beta^{*}\right)$\\
			\hline
			$1e+02$ &IT CPU RES RRES& 
			116\hspace{1.6cm} 0.92\hspace{1.6cm} 5.83e-06 \hspace{1.5cm}   8.15e-10&  65\hspace{1.6cm}
0.65\hspace{1.6cm} 6.24e-06\hspace{0.2cm}  8.72e-10
			\\
			\hline 
			$1e+03$&IT CPU RES RRES   &131 \hspace{1.8cm}1.07\hspace{0.8cm} 6.10e-04\hspace{1.8cm}    8.58e-10
 & 3 \hspace{1.8cm}0.11\hspace{0.8cm} 1.96e-04 \hspace{1.8cm}    2.89e-10
			\\
			\hline
	$1e+04$&IT CPU RES RRES   & 137 \hspace{1.8cm}1.10\hspace{0.8cm} 6.10e-04\hspace{1.8cm}    8.58e-10
 & 84 \hspace{1.8cm}0.86\hspace{0.8cm} 5.81e-04\hspace{1.8cm}    8.16e-10
\\
\hline
		\end{tabular}
\label{tab:v=1,exp4}
\end{table}

In the case $\nu =0.1.$\\
\begin{table}[H]
\centering
\caption{{Numerical results of  the $\mathcal{F}$GMRES methods.}}
		\begin{tabular}{ |p{1.4cm}||p{1.0cm}|||p{1.8cm}||p{1.8cm}|}
			\hline
			$\alpha$ & & $\mathcal{P}_{r}\left(\beta^{*}\right)$ &  $\mathcal{P}_{Gr}\left(\beta^{*}\right)$\\
			\hline
			$1e+02$ &IT CPU RES RRES& 
			102\hspace{1.6cm} 0.79\hspace{1.6cm} 6.58e-06 \hspace{1.5cm}   9.19e-10&  59\hspace{1.6cm}
0.64\hspace{1.6cm} 5.60e-06\hspace{0.2cm}  7.83e-10
			\\
			\hline 
			$1e+03$&IT CPU RES RRES   & 105\hspace{1.8cm} 0.77 \hspace{1.8cm} 5.85e-05\hspace{0.8cm}8.22e-10
 & 63\hspace{1.8cm}0.63\hspace{0.8cm} 5.73e-05\hspace{1.8cm}    8.05e-10
			\\
			\hline
	$1e+04$&IT CPU RES RRES   & 126\hspace{2.9cm}1.03\hspace{1.8cm}  5.50e-04\hspace{0.8cm} 7.73e-10
 &72\hspace{2.9cm}0.78\hspace{1.8cm}  7.05e-04\hspace{0.8cm} 9.90e-10
\\
\hline
		\end{tabular}
\label{tab:v=1,exp4}
\end{table}

From these experiments, it is evident that  our $\mathcal{P}_{Gr}\mathcal{F}$GMRES, in all trials, it necessitates a smaller amount of $\mathrm{CPU}$ time compared to other $\mathcal{F}$GMRES methods.
Tables $4$, $5$ and $6$, 
 demonstrate that using the $\mathcal{P}$GCG method to solve the algebraic linear system with multiple right-hand sides (\ref{PSMR}) as described in Algorithm $2$, enhances the convergence rate of  $\mathcal{F}$GMRES  method. The comparative analysis of simulation results in Tables  $4$, $5$ and $6$ indicates a substantial improvement in the performance of $\mathcal{F}$GMRES, considering CPU times, RES and RRES. This improvement is particularly notable when utilizing the  $\mathcal{P}_{Gr}$ preconditioner compared to the  $\mathcal{P}_{r}$ preconditioner. The diagonal $\mathcal{P}_{D}$ \cite{badahmane}, triangular $\mathcal{P}_{T}$ \cite{badahmane}, $\mathcal{P}_{RHSS}$ \cite{RHSS} and $\mathcal{P}_{LSS}$ \cite{LSS} preconditioners exhibits comparatively poorer performance.
\section{Conclusion}\label{conclusion}
In this paper, we have introduced  a  new  global regularized preconditioner $\mathcal{P}_{Gr}$ as described in (\ref{algor2}) for  saddle-point problem~(\ref{saddle}). This preconditioner improves the convergence rate of $\mathcal{R}$GMRES and $\mathcal{F}$GMRES  methods. As a result, the system of the generalized Stokes equations exhibits significant stiffness. The stiffness indicates that the saddle-point matrix has a large condition number, i.e. is ill-conditioned. In this situation, the $\mathcal{P}\mathcal{R}$GMRES method based on  the preconditioner $\mathcal{P}_{r}$ require an important $\mathrm{CPU}$ time. With  $\mathcal{P}_{Gr}$ preconditioner, the $\mathcal{P}\mathcal{R}$GMRES and  $\mathcal{F}$GMRES   methods converges very fast. Numerical experiments show that  the $\mathcal{P}_{Gr}\mathcal{R}$GMRES  method  with  suitable  parameter $\beta$, 
exhibits significant superiority over $\mathcal{P}_{r}\mathcal{R}$GMRES method  in terms of the  $\mathrm{CPU}$ time, and illustrate that the    $\mathcal{P}_{Gr}\mathcal{R}$GMRES and  $\mathcal{P}_{Gr}\mathcal{F}$GMRES  methods are a very efficient methods for solving the saddle-point problem. Future work will focus on the deterministic/Stokes  that is enabled in~\cite{koko} code when the simulated system is too large.
\bibliographystyle{elsarticle-num-names}

\end{document}